\documentclass[draft=false]{scrartcl}%

\usepackage{amsfonts}
\usepackage{amsmath}
\usepackage{amsthm}
\usepackage{cite}
\usepackage[T1]{fontenc}
\usepackage{hyperref}%[final]
\usepackage{mathdots}
\usepackage{mathtools}
\usepackage{microtype}
\usepackage[automark]{scrpage2}%

\usepackage{neccond_arxiv}

\mathtoolsset{mathic=true}

\title{A necessary condition for certain functions to preserve positive semi-definiteness on partitioned matrices}
\author{Lutz Klotz \and Conrad M\"adler}

\begin{document}
\maketitle
% \tableofcontents

\begin{abstract}
 If \(f\) is a symmetric complex-valued function on the \(m\)\nobreakdash-fold Cartesian product of the set of \tnn{} reals and \(A\) is a \tpsd{} \tmma{matrix} with eigenvalues \(\eigj\), we set \(f(A)\defeq f(\eige,\dotsc,\eigm)\). It is shown that if \(\mat{f(\Aab)}\) is \tpsd{} whenever \(\mat{\Aab}\) is a \tpsd{} matrix with \tpsd{} entries \(\Aab\), then \(f\) has a power series expansion with positive coefficients.
\end{abstract}

\begin{description}
 \item[Keywords] Positive semi-definite matrix, symmetric function, matrix function
 \item[AMS Classification] 15A15, 15B57
\end{description}

 The symbols \symba{\N}{n}, \symba{\Zp}{z}, \symba{\R}{r}, \symba{\Rp}{r}, and \symba{\C}{c} stand for the set of positive integers, \tnn{} integers, real, \tnn{} real and complex numbers, resp. Let \(m,n\in\N\). If \(S\) is a subset of the algebra of all \tmma{matrices} with complex entries, let \symba{\csme{n}{S}}{m} denote the set of all \tnna{matrices} with entries from \(S\) and \symba{\psde{n}{S}}{m} the set of all \tpsd{} matrices of \(\csme{n}{S}\). If \(S=\csm{1}=\C\), we shall write \symba{\csme{n}{S}\defeq\csmn}{m} and \symba{\psde{n}{S}\defeq\psdn}{m}. For a matrix \(A\in\csmm\), let \symba{\eigj}{l}, \(j\in\mn{1}{m}\), denote the eigenvalues of \(A\) counted according to their algebraic multiplicity and \symba{\det A}{d} its determinant. The vectors of \(\Cm\) are written as row vectors. If \(\mat{z_1,\dotsc,z_m}\in\Cm\), denote by \symba{\diag(z_1,\dotsc,z_m)}{d} the diagonal matrix with elements \(z_1,\dotsc,z_m\) on its principal diagonal. The symbol \symba{\Iu{m}}{i} stands for the unit matrix of \(\csmm\)
 
 If \(f\colon\Cm\to\C\) is a symmetric function, \ie{}\ if \(f(z_1,\dotsc,z_m)=f(z_{\pi(1)},\dotsc,z_{\pi(m)})\) for all \(\mat{z_1,\dotsc,z_m}\in\Cm\) and all permutations \(\pi\) of the set \(\mn{1}{m}\), we define \symb{f(A)\defeq f(\eige,\dotsc,\eigm)}, \(A\in\csmm\), according to the paper~\zita{MR0242855}. The function \(f\) is called \noti{to preserve \tpsd{ness} on \(\psde{n}{S}\)} if \(\mat{f(\Aab)}\in\psdn\) whenever \(\mat{\Aab}\in\psde{n}{S}\). It seems to be a difficult problem to characterize all functions preserving \tpsd{ness} on \(\psdb{n}{m}\) for given \(n>2\). If \(f\) is a polynomial, it was solved in~\zitaa{BGKP}{\cthm{1.1}}. On the other hand, it is not hard to describe all functions preserving \tpsd{ness} on \(\psdb{2}{m}\), \cf{}~\zitaa{MR3536943}{\cprop{1} and \cthm{2}}. As its consequence we can state the following result.
 
\begin{lemma}
 Let \(f\colon\Rpm\to\C\) be a symmetric function which preserves \tpsd{ness} on \(\psddb{2}{m}\). Then \(f\) is \tnn{}, increases with respect to each variable and is continuous.
\end{lemma}
\begin{proof}
 The first two properties of \(f\) can be obtained easily. Since \(f\) satisfies the inequality \(f(x_1y_1,\dotsc,x_my_m)^2\leq f(x_1,\dotsc,x_m)f(y_1,\dotsc,y_m)\) for all \(\mat{x_1,\dotsc,x_m},\mat{y_1,\dotsc,y_m}\in\Rpm\), \cf{}~\zitaa{MR3536943}{\cprop{1}}, its continuity can be shown by a straightforward modification of the proof of the conclusion~(a)\(\Rightarrow\)(b) of~\zitaa{MR2547900}{\clem{2.1}}.
\end{proof}

 We shall say that a symmetric function \(f\colon\Rpm\to\C\) \noti{preserves \tpsd{ness}} if it preserves \tpsd{ness} on \(\psddb{n}{m}\) for all \(n\in\N\). In the case \(m=1\), the problem to describe all functions preserving \tpsd{ness} and some related problems have a long history and were completely solved by several authors applying different methods, \cf{}~\zitas{MR747302,MR502895,MR0163181,MR0152832,MR2547900,MR0264736,MR0109204,MR0005922,MR537245}. A well known result is stated in \rthm{T2}.
 
\begin{theorem}\label{T2}
 A function \(f\colon\Rp\to\C\) preserves \tpsd{ness} if and only if it has a power series expansion \(f(x)=\sum_{j=0}^\infty a_jx^j\), \(x\in\Rp\), where \(a_j\in\Rp\), \(j\in\Zp\).
\end{theorem}

 Our \rthm{T5} below generalizes the ``only-if'' part of the preceding theorem to arbitrary \(m\in\N\).

 For \(p\in\Zp\), denote by \symba{\Conp}{k} the set of all \(\mat{p_1,\dotsc,p_m}\in\Zpm\) such that \(\sum_{\alpha=1}^mp_\alpha\leq p\). If \(\mat{p_1,\dotsc,p_m}\in\Conp\) and \(a_{\alpha\beta}\), \(\alpha\in\mn{1}{m}\), \(\beta\in\mn{1}{n}\), are positive real numbers, let \symba{\pvpm}{v} be a vector of \(\Rpm\), whose \(\beta\)\nobreakdash-th entry equals \(\prod_{\alpha=1}^ma_{\alpha\beta}^{p_\alpha}\).
 
\begin{lemma}\label{L3}
 Let \(m\in\N\), \(p\in\Zp\), and \(n\defeq\rk{p+2}^m\). There exist positive real numbers \(a_{\alpha\beta}\), \(\alpha\in\mn{1}{m}\), \(\beta\in\mn{1}{n}\), such that the corresponding vectors \(\pvpm\), \(\mat{p_1,\dotsc,p_m}\in\Conp\), are linearly independent.
\end{lemma}
\begin{proof}
 For \(n=\rk{p+2}^m\), let \(a_{1\beta}\) be \(n\) pairwise different positive real numbers and \(a_{\alpha\beta}\defeq a_{1\beta}^{\rk{p+2}^{\alpha-1}}\), \(\alpha\in\mn{1}{m}\), \(\beta\in\mn{1}{n}\). It is easy to see that for \(\mat{p_1,\dotsc,p_m}\in\Conp\), the transpose of the corresponding vector \(\pvpm\) is a column vector of the Vandermonde matrix \(\mat{a_{1\beta}^{\gamma-1}}_{\beta,\gamma=1,\dotsc,n}\). To finish the proof it is enough to show that \(\pvpm\neq\pvqm\) if \(\mat{p_1,\dotsc,p_m}\neq\mat{q_1,\dotsc,q_m}\), \(\mat{p_1,\dotsc,p_m},\mat{q_1,\dotsc,q_m}\in\Conp\). However, the equality \(\mat{p_1,\dotsc,p_m}=\mat{q_1,\dotsc,q_m}\) would imply that \(p+2\) is a zero of the polynomial \(P\): \(P(z)=\sum_{\alpha=1}^m\rk{p_\alpha-q_\alpha}z^\alpha\), \(z\in\C\), which is a contradiction to the estimate \(\abs{z_0}\leq\max\setaa{1+\abs{p_\alpha-q_\alpha}}{\alpha\in\mn{1}{m}}\) for any zero \(z_0\) of \(P\), \cf{}~\zitaa{MR2978290}{\cprob{5.6.26}}.
\end{proof}

 In the case \(m=1\) the result of the next lemma was proved by Horn, \cf{}\ the first part of the proof of~\zitaa{MR0264736}{\cthm{1.1}}. In 1979 Vasudeva~\zita{MR537245} published a simplified proof, which can be easily adapted to the case that \(m\) is an arbitrary positive integer.

\begin{lemma}\label{L4}
 Let \(f\colon\Rpm\to\Rp\) be a symmetric function, which has continuous partial derivatives of arbitrary order. If \(f\) preserves \tpsd{ness}, then all its partial derivatives are \tnn{}.
\end{lemma}
\begin{proof}
 Let \(\mat{x_1,\dotsc,x_m}\in\Rpm\) and \(n\in\N\). If \(\aab\) are positive real numbers and \(t_\alpha\in\Rpn\), \(\alpha\in\mn{1}{m}\), \(\beta\in\mn{1}{n}\), the matrix \(\mat{\diag\rk{x_1+t_1a_{1\beta}a_{1\gamma},\dotsc, x_m+t_ma_{m\beta}a_{m\gamma}}}_{\beta,\gamma=1,\dotsc,n}\) belongs to \(\psddb{n}{m}\). From the assumption on \(f\) it follows
\begin{align}\label{E1}
 \sum_{\beta,\gamma=1}^n\ko{z_\beta}z_\gamma f\rk{x_1+t_1a_{1\beta}a_{1\gamma},\dotsc, x_m+t_ma_{m\beta}a_{m\gamma}}&\geq0,&\mat{z_1,\dotsc,z_m}&\in\Cm.
\end{align}
 Expanding the function \(g_{\beta\gamma}\colon\Rpm\to\Rp\), \(g_{\beta\gamma}(t_1,\dotsc,t_m)\defeq f\rk{x_1+t_1a_{1\beta}a_{1\gamma},\dotsc, x_m+t_ma_{m\beta}a_{m\gamma}}\) into a Taylor polynomial of degree \(p-1\), from \eqref{E1} we obtain
\begin{multline}\label{E2}
 \sum_{\beta,\gamma=1}^n\ko{z_\beta}z_\gamma\Biggl\{\sum_{\mat{p_1,\dotsc,p_m}\in\Con{p-1}}\prod_{\alpha=1}^m\frac{1}{p_\alpha!}\rk{t_\alpha a_{\alpha\beta}a_{\alpha\gamma}}^{p_\alpha}\frac{\partial^{p_1+\dotsb+p_m}}{\partial t_1^{p_1}\dotsm\partial t_m^{p_m}}g_{\beta\gamma}(t_1,\dotsc,t_m)\\
 +\sum_{\mat{p_1,\dotsc,p_m}\in\Con{p}\setminus\Con{p-1}}\prod_{\alpha=1}^m\frac{1}{p_\alpha!}\rk{t_\alpha a_{\alpha\beta}a_{\alpha\gamma}}^{p_\alpha}\\
 \times\frac{\partial^{p_1+\dotsb+p_m}}{\partial t_1^{p_1}\dotsm\partial t_m^{p_m}}f\rk{x_1+\theta_{\beta\gamma}t_1a_{1\beta}a_{1\gamma},\dotsc, x_m+\theta_{\beta\gamma}t_ma_{m\beta}a_{m\gamma}}\Biggr\}
 \geq0,
\end{multline}
 where \(\theta_{\beta\gamma}\in[0,1]\), \(\beta,\gamma\in\mn{1}{n}\). Let \(\mat{q_1,\dotsc,q_m}\) be an arbitrary element of \(\Con{p}\setminus\Con{p-1}\). From \rlem{L3} it follows that for \(n=\rk{p+2}^m\), the positive numbers \(\aab\) can be chosen in such a way that there exist real numbers \(z_\beta\) satisfying the system of linear equations \(\sum_{\beta=1}^n\prod_{\alpha=1}^m\aab^{p_\alpha}z_\beta=0\) for \(\mat{p_1,\dotsc,p_m}\in\Conp\setminus\set{\mat{q_1,\dotsc,q_m}}\) and \(\sum_{\beta=1}^n\prod_{\alpha=1}^m\aab^{q_\alpha}z_\beta=1\). Therefore, \eqref{E2} implies that
 \begin{multline*}
  \prod_{\alpha=1}^m\frac{1}{q_\alpha!}\rk{t_\alpha}^{q_\alpha}\sum_{\beta,\gamma=1}^n\prod_{\alpha'=1}^m\ko{z_\beta}z_\gamma\rk{a_{\alpha'\beta}a_{\alpha'\gamma}}^{q_{\alpha'}}\\
  \times\frac{\partial^{q_1+\dotsb+q_m}}{\partial t_1^{q_1}\dotsm\partial t_m^{q_m}}f\rk{x_1+\theta_{\beta\gamma}t_1a_{1\beta}a_{1\gamma},\dotsc, x_m+\theta_{\beta\gamma}t_ma_{m\beta}a_{m\gamma}}\geq0,
 \end{multline*}
 which yields \(\frac{\partial^{q_1+\dotsb+q_m}}{\partial x_1^{q_1}\dotsm\partial x_m^{q_m}}f\rk{x_1,\dotsc,x_m}\geq0\) by letting \(\mat{t_1,\dotsc,t_m}\) tend to \(\mat{0,\dotsc,0}\).
\end{proof}

\begin{theorem}\label{T5}
 Let \(f\colon\Rpm\to\Rp\) be a symmetric function preserving \tpsd{ness}. Then \(f\) has a power series expansion
\begin{multline}\label{E3}
 f\rk{x_1,\dotsc,x_m}
 =\sum_{\mat{p_1,\dotsc,p_m}\in\Zpm}a_{p_1,\dotsc,p_m}x_1^{p_1}\dotsm x_m^{p_m},\\
 \mat{x_1,\dotsc,x_m}\in\Rpm,a_{p_1,\dotsc,p_m}\in\Rp,\mat{p_1,\dotsc,p_m}\in\Zpm.
\end{multline}
\end{theorem}
\begin{proof}
 Let \(\psi\colon\R\to\Rp\) have support on \((-1,0)\), continuous partial derivatives of arbitrary order and satisfy \(\int_\R\psi(t)\dif t=1\). Define \(\phi\rk{t_1,\dotsc,t_m}\defeq\prod_{j=1}^m\psi(t_j)\), \(\mat{t_1,\dotsc,t_m}\in\Rm\), and then for arbitrary \(\epsilon\in(0,\infty)\) the functions \(\phi_\epsilon(\bft)\defeq\phi(\bft/\epsilon)\), \(\bft\in\Rm\), and
\begin{align}\label{E4}
 f_\epsilon(\bfx)
 &\defeq\frac{1}{\epsilon^m}\int_{\Rm}f(\bft)\phi_\epsilon(\bfx-\bft)\dif\bft
 =\int_{(-1,0)^m}f(\bfx-\epsilon\bft)\phi_\epsilon(\bft)\dif\bft,&
 \bfx&\in\Rm.
\end{align}
The function \(f_\epsilon\) is symmetric, \tnn{}, and has continuous partial derivatives of arbitrary order. The integral at the right-hand side of \eqref{E4} is the limit of the integral sums of the form \(\sum_{j=1}^{r-1}f\rk{x_1-\epsilon\xi_j,\dotsc, x_m-\epsilon\xi_j}\phi(\xi_j)\rk{t_{j+1}-t_j}^m\), where \(-1=t_1<t_2<\dotsb<t_r=0\) is a partition of the interval \([-1,0]\) and \(\xi_j\in[t_j,t_{j+1}]\), \(j\in\mn{1}{r-1}\). If \(\mat{\Aab}\in\psddb{n}{m}\), then \(\mat{\Aab-\epsilon\xi_j\Iu{m}}\in\psddb{n}{m}\) and \(\eigk(\Aab-\epsilon\xi_j\Iu{m})=\eigk(\Aab)-\epsilon\xi_j\), hence \(\mat{\sum_{j=1}^{r-1}f\rk{\eige(\Aab)-\epsilon\xi_j,\dotsc,\eigm(\Aab)-\epsilon\xi_j}}_{\alpha,\beta=1,\dotsc,n}\) belongs to \(\psdn\) by assumption on \(f\), \(j\in\mn{1}{r-1}\). It follows \(\mat{f_\epsilon(\Aab)}\in\psdn\), and from \rlem{L4} we can conclude that all finite differences of \(f_\epsilon\) are \tnn{}, \cf{}~\zitaa{MR3215179}{\cpage{260}}. Since \(f\) is the pointwise limit of \(f_\epsilon\) if \(\epsilon\) tends to \(0\), all finite differences of \(f\) are \tnn{}. By~\zitaa{MR3215179}{\cthm{8.6}} the representation \eqref{E3} follows.
\end{proof}

 We conclude our paper with an application of the preceding result.
 
\begin{theorem}
 Let \(f\colon\Rpm\to\C\) be a symmetric function. The following assertions are equivalent:
 \begin{aeqi}{0}
  \il{T6.i} The function \(f\) has a power series expansion
\begin{align}\label{E5}
 f\rk{x_1,\dotsc,x_m}&=\sum_{j=0}^\infty b_j\rk{x_1\dotsm x_m}^j,&\mat{x_1,\dotsc,x_m}&\in\Rpm,b_j\in\Rp,j\in\Zp.
\end{align}
  \il{T6.ii} For all \(n\in\N\) and all \(\mat{\Aab},\mat{B_\ab}\in\psdb{n}{m}\) such that \(\mat{\Aab B_\ab}\in\csme{n}{\psd{m}}\), the matrix \(\mat{f(\Aab B_\ab)}\) is \tpsd{}.
 \end{aeqi}
\end{theorem}
\begin{proof}
 \begin{imp}{T6.i}{T6.ii}
  Assume that \(f\) has the form \eqref{E5}. Then \(f(\Aab B_\ab)=\sum_{j=0}^\infty b_j\rk{\det\Aab\det B_\ab}^j\) and the assertion~\ref{T6.ii} is a simple consequence of Schur's theorem and the well-known fact that \(\mat{\det\Aab}\in\psd{n}\) if \(\Aab\in\psdb{n}{m}\).
 \end{imp}

 \begin{imp}{T6.ii}{T6.i}
  Let \(B_\ab\defeq\Iu{m}\), \(\alpha,\beta\in\mn{1}{n}\). Since \(\mat{B_\ab}\in\psdb{n}{m}\) and \(\mat{f(\Aab B_\ab)}=\mat{f(\Aab)}\), from \rthm{T5} it follows that \(f\) has a power series expansion \eqref{E3}. Therefore, to prove~\ref{T6.i} it is enough to show that \(f\rk{x_1,\dotsc,x_{m-1},0}=f\rk{0,\dotsc,0}\) for all \(\mat{x_1,\dotsc,x_{m-1}}\in\Rpo{m-1}\). We can assume that \(f\) is not a constant function. Let \(r\in\mn{1}{m}\) be the smallest number such that there exists \(\mat{x_1,\dotsc,x_r}\in\Rpo{r}\) satisfying \(f(x_1^2,\dotsc,x_r^2,0,\dotsc,0)\neq f\rk{0,\dotsc,0}\). Since the coefficients of the power series expansion \eqref{E3} are \tnn{}, we get
\beql{E6}
 f(x_1^2,\dotsc,x_r^2,0,\dotsc,0)
 >f\rk{0,\dotsc,0}.
\eeq
  For \(\epsilon\in(0,\infty)\), define \(A_{11}^{(\epsilon)}=B_{22}^{(\epsilon)}\defeq\diag(x_1,\dotsc,x_r,\epsilon,\dotsc,\epsilon)\), \(A_{22}^{(\epsilon)}=B_{11}^{(\epsilon)}\defeq\diag(\epsilon,\dotsc,\epsilon,x_r,\dotsc,x_1)\), \(A_{12}^{(\epsilon)}=B_{21}^{(\epsilon)}\defeq\diag(x_1,\dotsc,x_r,\epsilon,\dotsc,\epsilon)J\), \(A_{21}^{(\epsilon)}=B_{12}^{(\epsilon)}\defeq\diag(\epsilon,\dotsc,\epsilon,x_r,\dotsc,x_1)\), where \(J\defeq\smat{0&&1\\&\iddots&\\1&&0}\). Then \(\mat{\Aab^{(\epsilon)}},\mat{B_\ab^{(\epsilon)}}\in\psdb{2}{m}\). If \(1\leq r\leq m/2\), we obtain \(f\rk{A_{11}^{(\epsilon)}B_{11}^{(\epsilon)}}=f\rk{A_{22}^{(\epsilon)}B_{22}^{(\epsilon)}}=f\rk{\epsilon x_1,\dotsc,\epsilon x_r,\epsilon x_r,\dotsc,\epsilon x_1,0,\dotsc,0}\) and \(f\rk{A_{12}^{(\epsilon)}B_{12}^{(\epsilon)}}=f\rk{A_{21}^{(\epsilon)}B_{21}^{(\epsilon)}}=f\rk{x_1^2,\dotsc,x_r^2,\epsilon^2,\dotsc,\epsilon^2}\). If \(\epsilon\) is small enough, from \eqref{E6} and the continuity of \(f\) it follows \(\det\mat{f\rk{\Aab^{(\epsilon)}B_\ab^{(\epsilon)}}}<0\) contradicting~\ref{T6.ii}. If \(m/2< r\leq m-1\), one obtains \(f\rk{A_{11}^{(\epsilon)}B_{11}^{(\epsilon)}}=f\rk{A_{22}^{(\epsilon)}B_{22}^{(\epsilon)}}=f\rk{x_{m-r+1}x_r,\dotsc,x_rx_{m-r+1},\epsilon x_1,\dotsc,\epsilon x_{m-r},\epsilon x_{m-r},\dotsc,\epsilon x_1}\) and \(f\rk{A_{12}^{(\epsilon)}B_{12}^{(\epsilon)}}=f\rk{A_{21}^{(\epsilon)}B_{21}^{(\epsilon)}}=f\rk{x_1^2,\dotsc,x_r^2,\epsilon^2,\dotsc,\epsilon^2}\). Since \(r\leq m-1\) yields \(2r-m\leq r-1\) and \(f\rk{y_1,\dotsc,y_{r-1},0,\dotsc,0}=f\rk{0,\dotsc,0}\) for all \(\mat{y_1,\dotsc,y_{r-1}}\in\Rpo{r-1}\), it follows \(f\rk{x_{m-r+1}x_r,\dotsc,x_rx_{m-r+1},0,\dotsc,0}=f\rk{0,\dotsc,0}\) and we again arrive at the contradiction \(\det\mat{f\rk{\Aab^{(\epsilon)}B_\ab^{(\epsilon)}}}<0\) if \(\epsilon\) is small enough.
 \end{imp}
\end{proof}

\paragraph{Acknowledgement}
 We are exceptionally grateful to Professor Fuzhen Zhang for pointing out Vasudeva's paper to us.

\bibliographystyle{abbrv}
\bibliography{neccond_arxiv}

\begin{thebibliography}{10}

\bibitem{BGKP}
A.~Belton, D.~Guillot, A.~Khare, and M.~Putinar.
\newblock Matrix positivity preservers in fixed dimension.~{I}.
\newblock {\tt arXiv:1504.07674v4 [math.CA]}, Apr. 2016.

\bibitem{MR747302}
C.~Berg, J.~P.~R. Christensen, and P.~Ressel.
\newblock {\em Harmonic analysis on semigroups}, volume 100 of {\em Graduate
  Texts in Mathematics}.
\newblock Springer-Verlag, New York, 1984.
\newblock Theory of positive definite and related functions.

\bibitem{MR502895}
J.~P.~R. Christensen and P.~Ressel.
\newblock Functions operating on positive definite matrices and a theorem of
  {S}choenberg.
\newblock {\em Trans. Amer. Math. Soc.}, 243:89--95, 1978.

\bibitem{MR0163181}
C.~S. Herz.
\newblock Fonctions op\'erant sur certains semi-groupes.
\newblock {\em C. R. Acad. Sci. Paris}, 255:2046--2048, 1962.

\bibitem{MR0152832}
C.~S. Herz.
\newblock Fonctions op\'erant sur les fonctions d\'efinies-positives.
\newblock {\em Ann. Inst. Fourier (Grenoble)}, 13:161--180, 1963.

\bibitem{MR2547900}
F.~Hiai.
\newblock Monotonicity for entrywise functions of matrices.
\newblock {\em Linear Algebra Appl.}, 431(8):1125--1146, 2009.

\bibitem{MR0264736}
R.~A. Horn.
\newblock The theory of infinitely divisible matrices and kernels.
\newblock {\em Trans. Amer. Math. Soc.}, 136:269--286, 1969.

\bibitem{MR2978290}
R.~A. Horn and C.~R. Johnson.
\newblock {\em Matrix analysis}.
\newblock Cambridge University Press, Cambridge, second edition, 2013.

\bibitem{MR3536943}
L.~Klotz and C.~M{\"a}dler.
\newblock Some functions preserving positive semidefiniteness of {$2\times 2$}
  block matrices.
\newblock {\em Linear Algebra Appl.}, 507:68--76, 2016.

\bibitem{MR0242855}
M.~Marcus and S.~M. Katz.
\newblock Matrices of {S}chur functions.
\newblock {\em Duke Math. J.}, 36:343--352, 1969.

\bibitem{MR3215179}
P.~Ressel.
\newblock Higher order monotonic functions of several variables.
\newblock {\em Positivity}, 18(2):257--285, 2014.

\bibitem{MR0109204}
W.~Rudin.
\newblock Positive definite sequences and absolutely monotonic functions.
\newblock {\em Duke Math. J}, 26:617--622, 1959.

\bibitem{MR0005922}
I.~J. Schoenberg.
\newblock Positive definite functions on spheres.
\newblock {\em Duke Math. J.}, 9:96--108, 1942.

\bibitem{MR537245}
H.~Vasudeva.
\newblock Positive definite matrices and absolutely monotonic functions.
\newblock {\em Indian J. Pure Appl. Math.}, 10(7):854--858, 1979.

\end{thebibliography}

\vfill\noindent
\begin{minipage}{0.5\textwidth}
 Universit\"at Leipzig\\
 Fakult\"at f\"ur Mathematik und Informatik\\
 PF~10~09~20\\
 D-04009~Leipzig
\end{minipage}
\begin{minipage}{0.49\textwidth}
 \begin{flushright}
  \texttt{
   klotz@math.uni-leipzig.de\\
   maedler@math.uni-leipzig.de
  } 
 \end{flushright}
\end{minipage}

\end{document}